\documentclass[12pt]{amsart}
\usepackage{amsmath,amssymb}
\usepackage{amsfonts}
\usepackage{amsthm}
\usepackage{latexsym}
\usepackage{graphicx}
\title{}

\def\lf{\left}
\def\ri{\right}

\def\Ric{\mbox{Ric}}

\def\la{\langle}
\def\ra{\rangle}

\def\p{\partial}

\def\ii{\sqrt{-1}}

\def\R{\mathbb{R}}
\def\C{\mathbb{C}}

\def\vv<#1>{\langle#1\rangle}

\def\XXint#1#2{\setbox0=\hbox{$#1{#2}{\int}$}{#2}\kern-.5\wd0 }

\def\XXint#1#2#3{{\setbox0=\hbox{$#1{#2#3}{\int}$}
     \vcenter{\hbox{$#2#3$}}\kern-.5\wd0}}

\def\vv<#1>{\langle#1\rangle}
\def\lf{\left}
\def\ri{\right}
\def\a{{\alpha}}
\def\b{{\beta}}
\def\g{{\gamma}}

\def\ba{{\bar\a}}
\def\bb{{\bar\b}}

\def\bg{{\bar\g}}

\def\abb{{\a\bb}}

\def\p{\partial}
\def\la{\langle}
\def\ra{\rangle}

\def\ii{\sqrt{-1}}

\def\R{{\mathbb{R}}}
\def\CP{{\mathbb{CP}}}

\def\be{\begin{equation}}
\def\ee{\end{equation}}

\def\lf{\left}
\def\ri{\right}
\def\a{{\alpha}}
\def\b{{\beta}}

\def\g{{\gamma}}

\def\ba{{\bar\a}}
\def\bb{{\bar\b}}

\def\bg{{\bar\g}}
\def\Ric{\text{Ric}}

\def\abb{{\a\bb}}

\def\p{\partial}

\def\C{\Bbb C}

\def\p{\partial}

\def\C{\Bbb C}
 \def\ii{\sqrt{-1}}

\newtheorem{thm}{Theorem}[section]

\newtheorem{lem}{Lemma}[section]
\newtheorem{prop}{Proposition}[section]
\newtheorem{cor}{Corollary}[section]
\theoremstyle{definition}
\newtheorem{defn}{Definition}[section]
\theoremstyle{remark}

\newtheorem{rem}{Remark}[section]

\numberwithin{equation}{section}

\begin{document}
\title{Some comparison theorems for K\"ahler manifolds}


\begin{abstract} In this work, we will verify some comparison results on K\"ahler manifolds. They are  complex Hessian comparison for   the distance function from a closed complex submanifold of a K\"ahler manifold with holomorphic bisectional curvature bounded below by a constant, eigenvalue comparison and volume comparison in terms of scalar curvature. This work is motivated by  comparison results of Li and Wang \cite{LW}.
\end{abstract}

\renewcommand{\subjclassname}{\textup{2010} Mathematics Subject Classification}
\author{Luen-Fai Tam$^1$ and Chengjie Yu$^2$}

\address{The Institute of Mathematical Sciences and Department of
 Mathematics, The Chinese University of Hong Kong,
Shatin, Hong Kong, China.} \email{lftam@math.cuhk.edu.hk}

\address{Department of Mathematics, Shantou University, Shantou, Guangdong, China } \email{cjyu@stu.edu.cn}

\thanks{$^1$Research partially supported by Hong Kong RGC General Research Fund  \#CUHK 403108}
\thanks{$^2$Research partially supported by the National Natural Science Foundation of
China (10901072) and GDNSF (9451503101004122).}
\date{October 2010}

\maketitle
\markboth{Luen-Fai Tam and Chengjie Yu}
 {Some comparison theorems for K\"ahler manifolds}
\section{Introduction}

In this work, we will study some comparison theorems on K\"ahler
manifolds. There is a well-known Hessian comparison for the distance function on Riemannian manifolds in terms of lower bound of sectional curvature. It is expected that for K\"ahler manifolds the lower bound of sectional curvature can be replaced by the lower bound of bisectional curvature to obtain a complex Hessian comparison for the distance function. In fact, in  \cite{LW}, Li and
Wang gave a sharp upper estimate for the Laplacian of the distance
function from a point. They also
gave a sharp upper estimate for the complex Hessian for the distance
function in the case that the lower bound of the bisectional curvature $K=0$, i.e., on K\"ahler manifolds with
nonnegative holomorphic bisectional curvature. This last result was
also proved by Cao and Ni \cite{CN} using Li-Yau-Hamilton Harnack
type inequality for the heat equation.   In the note, we shall verify this complex Hessian comparison for general lower bound $K$ and show that the complex Hessian of the distance function is bounded above by that in the complex space forms. See Theorem \ref{hesscomp-t1} and \ref{hesscomp-t2} for more details.

Our next result is motivated by the   well-known Licherowicz-Obata
theorem on Riemannian manifolds, which says that if $(M^m,g)$ is a
compact Riemannian manifold with Ricci curvature bounded below by
$(m-1)K$, where $K>0$ is a constant, then the first nonzero
eigenvalue $\lambda_1$ satisfies $\lambda_1\ge mK$ and equality
holds only if $(M,g)$ is isometric to the standard sphere of radius
$1/\sqrt K$. It is well-known that if $(M^n,g)$ is a K\"ahler
manifold such that Ricci curvature is such that $R_{\abb}\ge
kg_\abb$ for some constant $k>0$, then the first nonzero eigenvalue
$\lambda_1$ of the complex Laplacian is at least $k$ (see
\cite{Futaki1988}). Our next result is the following:

{\it Let $(M^n,g)$ be as above. Suppose the K\"ahler form is in the
first Chern class. If $\lambda_1=k$, then $M$ is K\"ahler-Einstein.}

As a corollary, when $(M^n,g)$ is a K\"ahler manifold with positive
holomorphic bisectional curvature such that the Ricci curvature is
bounded below by $ n+1 $, then $\lambda_1\ge n+1$ and equality holds
if and only if $(M^n,g)$ is holomorphically isometric to $\CP^n$
with the Fubini-Study metric (of constant holomorphic sectional
curvature 2). As an application we obtain a partial result on the equality case for the diameter estimate of
 Li and Wang \cite{LW} on compact K\"ahler manifolds with holomorphic bisectional curvature bounded below by 1. See Corollary \ref{diamter-t}.

 In \cite{LW}, it was proved that if a compact K\"ahler manifold has bisectional curvature bounded below by 2 (see Definition \ref{definition1}), then the volume  of the manifold is less than or equal to the volume of $\CP^n$ with the Fubini-Study metric. Our last result is an observation to relax their conditions by replacing the lower bound of the bisectional curvature by the lower bound of the scalar curvature,  with the assumption that the bisectional curvature is positive. In fact, we prove:

{\it Let $(M^n, g)$ be a compact manifold with positive holomorphic bisectional curvature. Suppose the scalar curvature $k_2\ge R\ge k_1>0$ for some constants $k_1$ and $k_2$. Then the volume $V(M,g) $ of $M$ satisfies $V(\CP^n,h_{k_2})\le V(M,g)\le V(\CP^n,h_{k_1})$. Moreover, if   one of the inequalities is an equality, then  $(M,g)$ is holomorphically isometric the $ \CP^n $ with the Fubini-Study metric.}

The result is a consequence of a more general result,  which is related to a conjecture of Schoen \cite{Schoen89} which states: If $(M^n,h)$ is a closed hyperbolic manifold and $g$ is another metric on $M$ with scalar curvature $R(g)\ge R(h)$, then $V(g)\ge V(h)$. The conjecture was proved for $n=3$ by Perelmann \cite{Perelman02,Perelman03}. We will compare volume of a compact K\"ahler manifold with the volume of a related  K\"ahler-Einstein metric in terms of upper and lower bounds of the scalar curvature. See Proposition \ref{volume}.

The paper is organized as follows: In \S2, we will study comparison of the complex Hessian for the distance function. In \S3, we will study eigenvalue comparison and in \S4, we will study volume comparison.

The authors would like to thank  Xiaowei Wang for some useful discussions.

\section{Complex Hessian comparisons}

Let $ M^n$ be a complex manifold with complex dimension $n$ and let $J$ be the complex structure. Suppose $g$ is  a Hermitian metric such that $(M,J,g)$ is K\"ahler. Suppose $\{e_1,e_2,\cdots,e_n\}$
 is a frame on $T^{(1,0)}(M)$, let $g_{\alpha\bar \beta}:=g(e_{\a},\bar
 e_\beta)$. We also write $g(X,\bar Y)$ as $\la X,  \bar Y\ra$ and $||X||^2=g(X,\bar X)$ for $X, Y\in T^{(1,0)}(M)$.
 Following \cite{LW}, we have the following definition.
 \begin{defn}\label{definition1} Let $(M,J, g)$ be a K\"ahler manifold. We say that the holomorphic bisectional curvature of $M$ is bounded below by a constant $K$, denoted by $BK_M\ge K$, if
 \begin{equation}\label{holobi-def}
 \frac{R(X,\bar X,Y,\bar Y)}{\|X\|^2\|Y\|^2+|\vv<X,\bar Y>|^2}\geq K
 \end{equation}
 for any two nonzero (1,0)-vectors $X,Y$.
 \end{defn}
 We remark that on a K\"ahler manifold with constant holomorphic sectional curvature $2K$, \eqref{holobi-def} is an equality for all nonzero $X, Y\in T^{(1,0)}(M)$.

 Let $S^p$ be a complex submanifold of $M^n$ with complex dimension $p\ge 0$ and let $r$ be the distance function from $S$. We want to give a comparison result for  the complex Hessian of $r$. We always assume that $S$ is connected and closed. Let $x_0\in S$ and let $\sigma$ be a geodesic from $x_0$ which is orthogonal to $S$ at $x_0$ and is parametrized by arc length $t$ with $0\le t\le T$. Assume that $\sigma$ is within the cut locus of $S$ so that the distance function $r$ is smooth near $\sigma|_{(0,T]}$. Let $e_1=\frac1{\sqrt 2}(\sigma'-\ii J\sigma')=\frac1{\sqrt 2}(\nabla r-\ii J \nabla r )$ and let $e_1,e_2,\dots, e_n$ be unitary frames  parallel along $\gamma$ such that $e_{n-p+1},\dots,e_n$ are tangent to $S$ at $t=0$. In the following   $f_{\abb}$ etc are   covariant derivatives of the function $f$  and repeated indices mean summation. The computations in \cite{LW} show:

 \begin{lem}\label{evolutionhess-l1} With the above notations, on $\sigma$, we have
 \begin{equation}\label{evolutionhess-e1}
   \frac{d}{dr}r_{\abb}+r_{\a\bg}r_{\g\bb}+r_{\a\g}r_{\bg\bb}=-\frac12R_{\abb1\bar1},
 \end{equation}
 and
 \begin{equation}\label{evolutionhess-e2}
   \frac{d}{dr}r_{\a\b}+r_{\a\g}r_{\bg\b}+r_{\a\bg}r_{\g\b}= \frac12R_{\a\bar1\b\bar1}.
 \end{equation}
 \end{lem}
 \begin{proof} Extend $e_i$ to be a unitary frames near a point on $\sigma$.
Note that on $\sigma$
\begin{equation}
\begin{split}
0=&(\|\nabla r\|^2)_{\alpha\bar\beta}\\
=&(2r_{\gamma}r_{\bar\gamma})_{\alpha\bar\beta}\\
=&2(r_{\gamma\alpha\bar\beta}r_{\bar\gamma}+r_{\gamma}r_{\bar\gamma\alpha\bar\beta}+r_{\gamma\alpha}r_{\bar\gamma\bar\beta}+r_{\gamma\bar\beta}r_{\bar\gamma\alpha})\\
=&2(r_{\alpha\bar\beta\gamma}r_{\bar\gamma}+R_{\alpha\bar\delta\gamma\bar\beta}r_{\delta}r_{\bar\gamma}+r_{\gamma}r_{\alpha\bar\beta\bar\gamma}+r_{\gamma\alpha}r_{\bar\gamma\bar\beta}+r_{\gamma\bar\beta}r_{\bar\gamma\alpha}).
\end{split}
\end{equation}
So,
\begin{equation}\label{eqn-1}
r_{\alpha\bar\beta\gamma}r_{\bar\gamma}+r_{\alpha\bar\beta\bar\gamma}r_{\gamma}+r_{\alpha\gamma}r_{\bar\gamma\bar\beta}+r_{\gamma\bar\beta}r_{\alpha\bar\gamma}=-R_{\alpha\bar\delta\gamma\bar\beta}r_{\delta}r_{\bar\gamma}.
\end{equation}
Since $e_1=\frac{1}{\sqrt 2}(\nabla r-J\nabla r)$, and $e_1,e_2,\cdots,e_n$ are parallel   along $\sigma$, we have
\begin{equation}\label{eqn-2}
r_{1}=\frac{1}{\sqrt{2}}\ \mbox{and}\ r_{\alpha}=0\ \mbox{for any $\alpha>1$ }
\end{equation}
and
\begin{equation}\label{eqn-3}
 r_{\alpha\bar\beta\gamma}r_{\bar\gamma}+r_{\alpha\bar\beta\bar\gamma}r_{\gamma}
 =\frac1{\sqrt 2}(r_{\alpha\bar\beta 1}+r_{\alpha\bar\beta\bar1})=\frac{d}{dr}r_{\alpha\bar\beta}.
 \end{equation}
By \eqref{eqn-1}--\eqref{eqn-3}, we conclude that \eqref{evolutionhess-e1} is true.

The proof of \eqref{evolutionhess-e2} is similar.
 \end{proof}

 The following fact may be well-known.  We include the proof for the sake of completeness.

\begin{lem}\label{lem-asym-r-S}
Let $M^{n+m}$ be a complete Riemannian manifold and $S^{m}$ be a
closed submanifold of $M$. Let $r$ be the distance function to $S$.
Let $\gamma(s)$ be a normal geodesic orthogonal to $S$ with
$\gamma(0)\in S$.  Let $e_1,e_2,\cdots,e_n,e_{n+1},\cdots,e_{n+m}$
are parallelled orthonormal frames along $\gamma$ such that
$e_{n+1}(0),e_{n+2}(0),\cdots,e_{n+m}(0)$ is tangent to $S$ and
$e_1=\gamma'$. Then
\begin{equation}
\lim_{s\to 0}\Bigg(\lf[\nabla^2r(\gamma(s))(e_i,e_j)\ri]_{1\leq i,j\leq
n+m}-\left(\begin{array}{ccl}0&0&0\\0&\frac{1}{s}I_{n-1}&0\\0&0&
[h_{ij}]_{n+1\leq
i,j\leq n+m}
\end{array}\right)\Bigg)=0.
\end{equation}
where $h_{ij}=\la h(e_i,e_j),\gamma'\ra$, and $h(X,Y)=-(\nabla_XY)^\perp$ for $X, Y\in T(S)$ is the second fundamental form of $S$.
\end{lem}
\begin{proof}
Let $p=\gamma(0)$ and let $\{\nu_{1},\dots,\nu_{n}\}$ be a unit normal frame of
the normal bundle $T^\perp S$ near $p$ such that $D\nu_i=0$ at $p$ for
$i=1,\cdots,n$, where $D$ is the normal connection given by $D_X \nu=(\nabla_X\nu)^\perp$ for any normal vector field $\nu$ and any $X$ tangent to $S$. Moreover, we may choose $\nu_i$ such that
$\nu_i(p)=e_i(0)$ for $i=1,2,\cdots,n$.

Now let $x_1,\dots,x_m$ be a local coordinates at $p$ in $S$, then
one can parametrize $M$  by $F(x,y)=\exp_{x}(\sum_{j=1}^n y_j\nu_j)$,
where $x=(x_1,\dots,x_m)$, $y=(y_1,\dots,y_n)$. Let
$X_i=\frac{\p}{\p x^i}$, $Y_j=\frac{\p }{\p y_j}$. Then
$r^2(x,y)=\sum_{j=1}^n y_j^2$. We can moreover assume that
$X_i(p)=e_{i+n}(0)$ for $i=1,2,\cdots,m$. Similar to normal
coordinates, we have
\begin{equation}\label{eqn-con-1}
 \nabla_{Y_i}Y_j =0
\end{equation}
on $S$. Extend $e_i$'s to be smooth vector fields near $p$. For any $i$,
$$
e_i=\sum_{k=1}^m a^k_iX_k+\sum_{l=1}^nb^l_i Y_l.
$$
Note that $a^k_i=0$ and $b^l_i=\delta^l_i$ at $s=0$ for $1\le i\le
n$, and $a^k_i=\delta^k_{i-n}$ and $b^l_i=0$ at $s=0$ for $n+1\le
i\le m+n$.

Since $e_i$ is parallel along   $\gamma$,
 by \eqref{eqn-con-1} and the initial conditions of $a^k_i$ and $b^l_i$, we conclude that if $1\le i\le n$, then
 \begin{equation}\label{eqn-con-3}
 a^k_i=o(s), b^l_i-\delta^l_i=o(s)
 \end{equation}
 as $s\to 0$. Using \eqref{eqn-con-1}, the fact that $D_{X_k}Y_1=0$ at $s=0$ and the fact that $X_k$ are tangential at $s=0$,
    if $n+1\le i\le m+n$, then
  \begin{equation}\label{eqn-con-4}
 b^l_i=o(s)
 \end{equation}
as $s\to0$. It is easy to see that on $\gamma$, $||\nabla^2 r||=O(\frac1s)$.  Hence by \eqref{eqn-con-3} for $1\le i, j\le  n$, we have on $\gamma$
\begin{equation}\label{eqn-con-5}
\begin{split}
0=&\lim_{s\to 0}\lf[\nabla^2 r(\gamma(s))(e_i,e_j)-\nabla^2r(\gamma(s))(Y_i,Y_j)\ri]\\
=&\lim_{s\to
0}\Big(\nabla^2r(\gamma(s))(e_i,e_j)-\frac{\delta_{ij}-\delta_{i1}\delta_{1j}}{s}\Big)
\end{split}
\end{equation}
where we have used \eqref{eqn-con-1} and the facts that on $\gamma$, $y_1=s>0$, $y_j=0$ for $2\le j\le n$, and $\nabla r=\gamma'$.

By  \eqref{eqn-con-3} and \eqref{eqn-con-4}, for $n+1\le i\le m+n$ and $ 1\le j\le  n$, we have on $\gamma$

\begin{equation}\label{eqn-con-6}
\begin{split}
\lim_{s\to 0}\nabla^2 r (e_i,e_j)=&\lim_{s\to
0}\sum_{k=1}^ma^k_i\nabla^2r (X_k,Y_j)\\
=&0
\end{split}
\end{equation}
where we have used the fact $r$ does not depend on $x$ and the fact that $DY_j=0$ at $p=\gamma(0)$.

By \eqref{eqn-con-4}, for $n+1\leq i,j\leq n+m$, on $\gamma$ we have
\begin{equation}\label{eqn-con-7}
\begin{split}
\lim_{s\to 0}\nabla^2 r (e_i,e_j)=&\lim_{s\to 0}
\sum_{k,l=1}^ma^k_ia^l_j\nabla^2r(X_k,X_l)\\
=& \la h(e_i,e_j),\gamma'(0)\ra
\end{split}
\end{equation}
where $h(X_i,X_j)=-\lf(\nabla_{X_i}X_j\ri)^\perp$ is the second
fundamental form of $S$. The completes the proof of the lemma.

\end{proof}

We also need the following result in \cite{EH}:
\begin{lem}\label{lem-Riccati-comp}
Let $A(t)$ and $B(t)$ be two smooth curves in the space of $n\times
n$ complex Hermitian matrices on the interval $(0,T)$, satisfying
the equality
\begin{equation}
A'(t)+A^2(t)\leq B'(t)+B^2(t)
\end{equation}
for any $t\in (0,T)$. Suppose
\begin{equation}
\lim_{t\to 0^+} (B(t)-A(t))=0.
\end{equation}
Then
\begin{equation}
A(t)\leq B(t)
\end{equation}
for $t\in (0,T)$. Here $X\le Y$ means that $Y-X$ is positive semi definite.
\end{lem}

\begin{thm}\label{hesscomp-t1}
Let $(M^n,g)$ be a complete K\"ahler manifold with holomorphic
bisectional curvature bounded below by $K$. Let   $S^p$ be a
connected closed complex submanifold of $M$ with complex dimension
$p$. Then within  the cut-locus of $S$,
\begin{equation}\label{eqn-complex-Hess-compare-1}
r_{\abb}\le F_K
\lf(g_{\alpha\bar\beta}-g^{S}_{\alpha\bar\beta}\ri)+G_K r_\a
r_{\bb}+H_Kg^{S}_{\alpha\bar\beta}
 \end{equation}
 where
 \begin{equation}\label{eqn-complex-Hess-compare-2}
 F_K=\left\{\begin{array}{clcr} \sqrt{\frac K2}\cot(\sqrt {\frac K2}r),   &\text{ if  $K>0$;}\\
  \frac1r,     &\text{ if  $K=0$;}\\
   \sqrt{-\frac K2}\coth(\sqrt {-\frac K2}r),   &\text{ if  $K<0$,}\\
 \end{array}\right.,
 \end{equation}
 \begin{equation}\label{eqn-complex-Hess-compare-3}
 G_K=\left\{\begin{array}{clcr}\sqrt{2K}\lf(\cot(\sqrt {2K}r)-\cot(\sqrt {\frac K2}r)\ri) \ \  &\text{ if  $K>0$;}\\
 -\frac1r \ \  &\text{ if  $K=0$;}\\
  \sqrt{-2K}\lf(\coth(\sqrt {-2K}r)-\coth(\sqrt {-\frac K2}r\ri) \ \  &\text{ if  $K<0$,}\\
 \end{array}\right.
 \end{equation}
 and
 \begin{equation}\label{eqn-complex-Hess-compare-4}
 H_K=\left\{\begin{array}{clcr}-\sqrt{K/2}\tan{(\sqrt {K/2}r)} \ \  &\text{ if  $K>0$;}\\
0 \ \  &\text{ if  $K=0$;}\\
  -\sqrt{-K/2}\tanh{(\sqrt {-K/2}r)} \ \  &\text{ if  $K<0$,}\\
 \end{array}\right.
 \end{equation}
  where $g^S$ is the metric of $S$ parallel
 transported along geodesics emanating from $S$ that is orthogonal
 to $S$.
\end{thm}
\begin{proof}
We only prove the case   $K>0$ and the   proofs
of the other  two cases are similar.

Let $x_0\in S$ and let $\gamma$ be a geodesic emanating from $x_0$
orthogonal to $S$. Let $e_1,e_2,\cdots,e_n$ be parallel unitary
frame along $\gamma$ with
$e_1=\frac{1}{\sqrt{2}}(\gamma'-\sqrt{-1}J\gamma')$ and
$e_{n-p+1}(0),\cdots,e_{n}(0)$ be tangent to $S$. Then, by Lemma
\ref{lem-asym-r-S} and the fact that $h(u,\bar v)=0$ for $u, v\in T^{1,0}S$,
\begin{equation}\label{initial-e1}
\lim_{r\to 0}\Bigg((r_{\alpha\bar\beta})_{1\leq\alpha,\beta\leq n}-\left(\begin{array}{ccc}\frac{1}{2r}&0&0\\
0&\frac{1}{r}I_{n-p-1}&0\\0&0&0\\
\end{array}\right)\Bigg)=0.
\end{equation}
Since $e_1=\frac{1}{\sqrt 2}(\nabla r-\sqrt{-1}J\nabla r)$,
\begin{equation}\label{eqn-2}
r_{1}=\frac{1}{\sqrt{2}}\ \mbox{and}\ r_{\alpha}=0\ \mbox{for any $\alpha>1$.}
\end{equation}
Moreover, since $(\|\nabla r\|^2)_\alpha=0$, we have
\begin{equation}\label{eqn-4}
r_{\alpha1}=-r_{\alpha \bar 1}.
\end{equation}
Therefore, by Lemma \ref{evolutionhess-e1}, we have
\begin{equation}\label{eqn-hessian}
\frac{d}{dr}r_{\alpha\bar\beta}+r_{\alpha\bar\gamma}r_{\gamma\bar\beta}+r_{\alpha\bar 1}r_{1\bar\beta}+\sum_{\gamma=2}^{n}r_{\alpha\gamma}r_{\bar\gamma\bar\beta}=-\frac{1}{2}R_{\alpha\bar\beta1\bar1}.
\end{equation}
Let $B=(r_{\alpha\bar\beta})$. Since the holomorphic bisectional curvature is bounded below by $K$, we conclude that
\begin{equation}
\frac{d}{dr}B+BB^*+B_1B_1^{*}\leq\left(\begin{array}{cc}-K&0\\
0&-\frac{K}{2}I_{n-1}\\
\end{array}\right).
\end{equation}
where $B_1$ is the first column of $B$. This implies that

\begin{equation}
\frac{d}{dr}\tilde B+\tilde B^2\leq\left(\begin{array}{cc}-2K&0\\
0&-\frac{K}{2}I_{n-1}\\
\end{array}\right)
\end{equation}
where $\tilde B=DBD$ and   $D=\left(\begin{array}{cc}\sqrt 2&0\\
0&I_{n-1}\\
\end{array}\right).$   By \eqref{initial-e1},
 \begin{equation}\label{initial-e2}
\lim_{r\to 0}\Bigg(\tilde B-\left(\begin{array}{ccc}\frac{1}{r}&0&0\\
0&\frac{1}{r}I_{n-p-1}&0\\0&0&0\\
\end{array}\right)\Bigg)=0.
\end{equation}
Note that
\begin{equation}
X=\left(\begin{array}{ccc}\sqrt{2K}\cot(\sqrt {2K} r)&0&0\\
0&\sqrt{\frac K2}\cot(\sqrt {\frac K2} r)I_{n-p-1}&0\\
0&0&-\sqrt{\frac K2}\tan(\sqrt{\frac K2}r)I_p
\end{array}\right)
\end{equation}
is a solution of
  \begin{equation}
\frac{d}{dr}X+X^2=\left(\begin{array}{cc}-2K&0\\
0&-\frac{K}{2}I_{n-1}\\
\end{array}\right)
\end{equation}
with initial condition \eqref{initial-e2}. By Lemma
\ref{lem-Riccati-comp}, $\tilde B\le X$. By the definition of
$\tilde B$, we conclude that the theorem is true for the case $K>0$.
The other cases are similar.
\end{proof}
In case $S$ is  a point $x_0\in M$, then it is understood that $g^S$ is zero. So within the cut locus of $x_0$, we have
$$
r_{\abb}\le F_K g_{\abb}+G_Kr_\a r_{\bb}.
$$

Next, we want to discuss the equality case.

\begin{thm}\label{hesscomp-t2} With the same assumptions in the Theorem \ref{hesscomp-t1}.
\begin{enumerate}
  \item [(i)] Suppose equality holds in \eqref{eqn-complex-Hess-compare-1} in a neighborhood of $S$, then $S$ is totally geodesic.
  \item [(ii)] If $S=x_0$ is a point and equality holds in \eqref{eqn-complex-Hess-compare-1} at all points in  the geodesic ball $B_{x_0}(r)$ which are within the cut locus of $x_0$, then $B_{x_0}(r)$ is holomorphically isometric to the geodesic ball of radius $r$ in the simply connected K\"ahler manifold with constant holomorphic sectional curvature $2K$. Here if $K>0$, it is assumed that $r<\frac{\pi}{\sqrt{2K}}$.
\end{enumerate}
\end{thm}

\begin{proof} The proof is similar to the Riemannian case (see
\cite{Klingenberg,CE}).

 (i)  When equality of \eqref{eqn-complex-Hess-compare-1} holds near $S$, we know
that
\begin{equation}\label{equality-e1}
r_{\alpha\beta}=0
\end{equation}
for all $\alpha, \beta\ge2$ near $S$. By Lemma \ref{lem-asym-r-S} and the fact that $h(u,\bar v)=0$ for $u, v\in T^{1,0}S$, we know that the second fundamental form $h$ of $S$ is
zero, hence $S$ is totally geodesic.

 (ii)  We only prove the case $K>0$, the other two cases are similar. Under the assumptions of the theorem, at all points within the cut locus of $x_0$,  we have
\begin{equation}\label{eqn-curvature}
(R_{\alpha\bar \beta 1\bar 1})=\left(\begin{array}{cc}2K&0\\
0&KI_{n-1}\\
\end{array}\right)
\end{equation}
and
\begin{equation}\label{eqn-hessian-r2}
 (r_{\alpha\beta})= \left(\begin{array}{cc}-\frac{\sqrt{2K}}{2}\cot(\sqrt {2K} r)&0\\
0&0_{n-1}\\
\end{array}\right).
\end{equation}
By Lemma \ref{evolutionhess-l1}, we have
\begin{equation}\label{eqn-hessian-2}
\frac{d}{dr}r_{\alpha\beta}+r_{\gamma\alpha}r_{\bar\gamma\beta}+r_{\gamma\beta}r_{\bar\gamma\alpha}=R_{\alpha\bar\delta\beta\bar\gamma}r_{\delta}r_{\gamma}=\frac{1}{2}R_{\alpha\bar
1\beta\bar 1}.
\end{equation}
\begin{equation}\label{eqn-curvature-2}
R_{\alpha\bar1\beta\bar 1}=0 \mbox{ if $\alpha\neq 1$ or $\beta\neq
1$}
\end{equation}
by substituting the equality in \eqref{eqn-complex-Hess-compare-1}  and \eqref{eqn-hessian-r2}
into \eqref{eqn-hessian-2}. Let $J=x_{\alpha}e_\alpha+\bar x_{\beta}\bar e_\beta$ be a Jacobi
field along a geodesic emanating from $p$. Then
\begin{equation}
x''_{\alpha}=\vv<R(\frac{e_1+\bar e_1}{\sqrt
2},x_{\beta}e_\beta+\bar x_{\gamma}\bar e_\gamma)\frac{e_1+\bar
e_1}{\sqrt 2},\bar e_\alpha>=\frac{1}{2}(R_{1\bar \gamma 1\bar
\alpha}\bar x_\gamma-R_{\beta\bar \alpha 1\bar 1}x_\beta)
\end{equation}
This is the same as the Jocobi field equation on the space form with
constant holomorphic  sectional curvature $2K$ because of
\eqref{eqn-curvature} and \eqref{eqn-curvature-2}. Let $o\in \CP^n$
with K\"ahler metric with constant holomorphic sectional curvature
$2K$. Let $I$ be an isometry from $T_{x_0}(M)\to T_o(\CP^n)$ such
that $I$ is holomorphic. Then suppose $r_0<$ injectivity radius of
$x_0$ and $r_0<r$, then the map $\exp_o\circ I\circ \exp_{x_0}^{-1}
:B_{x_0}(r_0)\to B_o(r_0)$ is an isometry by the proof of
Catan-Ambrose-Hicks Theorem \cite{Klingenberg}. Since the
injectivity radius of $o$ is $\frac{\pi}{\sqrt{2K}}$, it is easy to
see that $\phi=\exp_o\circ I\circ \exp_{x_0}^{-1} $ can be extended
to be an isometry from $B_{x_0}(r)$ to $B_o(r)$. We may arrange so
that $ \phi^*J_0=J_M$ at $x_0$ where $J_0$ is the complex structure
of $\CP^n$.  Since $\phi$ is an isometry, $\nabla(J_M-\phi^*J_0)=0$.
Hence $\phi$ is holomorphic.

\end{proof}

\begin{rem} By the theorem   and the proofs in \cite{LW}, we may obtain   results on the equality case for the volume
comparison and first Dirichlet eigenvalue comparison for geodesic
balls in \cite{LW}.
\end{rem}

\section{Eigenvalue comparison}

Consider the first nonzero eigenvalue of compact K\"ahler manifold
$(M^n,g)$. It is well known that if the Ricci curvature of $M$
satisfies $\Ric \ge k g$ for some $k>0$, then the first nonzero
eigenvalue of the complex Laplacian is bounded below by $k$ (see
\cite[Theorem 2.4.5]{Futaki1988}). Here the complex Laplacian of a
function $u$ is given by
$$\Delta u=g^{\abb}u_{\a \bb}.
$$
Let $\Delta^\R$ be the Laplacian of the underlining Riemannian metric, then $\Delta=\frac 12\Delta^\R$. We want to discuss the equality case.

\begin{thm} Let $(M^n,g)$ be a compact K\"ahler manifolds  with $\Ric \ge kg$, $k>0$. Suppose the first nonzero eigenvalue $\lambda$ for the complex Laplacian is $k$ and that the Ricci form $\rho$ and   K\"ahler form $\omega$ satisfy $[\rho]=c[\omega]$ for some $c$. Then $M$ is K\"ahler-Einstein.
 \end{thm}
\begin{proof}
   Let us recall the proof that $\lambda\ge k$.
Let $u$ be a first eigenfunction. If $\{e_\a\}$ is a unitary frame, then
 Then
\begin{equation}
\begin{split}
-\lambda\int_M|\nabla u|^2=&-\lambda \int_M \sum_\a u_\a u_{\bar \a}\\
=&\int_M \sum_\a (\Delta u)_\a u_{\bar \a}\\
=&\int_M \sum_{\a,\b} \lf( u_{\b\bb}\ri)_\a u_{\bar \a}\\
=&\int_M \sum_{\a,\b} \lf (u_{\b\a\bb}-R_{\b\bg \a\bb}u_\g\ri)u_{\ba}\\
=&-\lf(\int_M u_{\a\b}u_{\ba\bb}+R_{\a\bg}u_\g u_{\ba}\ri)
\end{split}
\end{equation}
Hence
$$
\lambda\int_M|\nabla u|^2\ge k\int_M|\nabla u|^2
$$
and $\lambda\ge k$.
Hence if $\lambda=k$, then $u_{\a\b}=0$. By assumption, there is a positive number $c>0$ and a function $F$
such that
 \begin{equation}\label{potential-e1}
   R_{\abb}-cg_{\abb}=F_{\abb}.
    \end{equation}

    Let $\phi=\Delta u+cu+g^{\abb}F_\a u_{\bb}$. As in
\cite[p.42]{Futaki1988}, one can prove that $\phi$ is holomorphic.
In fact, in a normal coordinate at a point,

\begin{equation}\label{holomorphic-e1}
    \begin{split}
    \phi_{\bg}=& u_{\a\ba\bg}+cu_{\bg}+F_{\a\bg}u_{\ba }+F_\a u_{\ba\bg}\\
    =&u_{\ba\a\bg}+cu_{\bg}+F_{\a\bg}u_{\ba }\\
    =&u_{\ba\bg\a }-R_{\a\bg}u_{\bar \a} +cu_{\bg}+F_{\a\bg}u_{\ba }\\
    =& \lf(-c g_{\a\bg}-F_{\a\bg}\ri) u_{\ba}+cu_{\bg}+F_{\a\bg}u_{\ba }\\
    =&0,
    \end{split}
\end{equation}
where we have used the fact that $u_{\a\b}=0$ and
\eqref{potential-e1}. Hence $\phi=a$ is a constant. This implies,
together with the fact that  $\Delta u=-ku$,
$$
(-k+c)u+g^{\abb}F_\a u_{\bb}=a.
$$
Let $S=\max u$, and $s=\min u$. Evaluate the above equality at the
maximum and the minimum points, we have
$$
(-k+c)S=a, (-k+c)s=a.
$$
So $k=c$ and $a=0$ because $u$ is nonconstant. Hence \eqref{potential-e1} becomes:

$$
R_{\abb}-kg_{\abb}=F_{\abb}.
$$
Since $R_{\abb}\ge kg_{\abb}$, $F$ is plurisubarmonic and is a
constant. Hence
$$
R_{\abb}-kg_{\abb}=0
$$
and $M$ is K\"ahler-Einstein.

\end{proof}
\begin{cor}\label{cor-first-eigen}
Let $(M^n,g)$ be a K\"ahler manifold with positive holomorphic bisectional
curvature such that $\Ric\ge kg$. Suppose  the first nonzero eigenvalue of the complex
Laplacian equal to $k$, then   $M$ is holomorphically isometric to $\CP^n$ with the Fubini-Study metric.
\end{cor}
\begin{proof} Since
  $b_2(M)=1$ (see \cite {BishopGoldberg1965,GoldbergKobayashi1967}),   the condition $[\rho]=c[\omega]$ in the   theorem is
automatically satisfied. Hence $g$ is K\"ahler-Einstein and by the
theorem of Berger (see \cite{Berger,GoldbergKobayashi1967}) $M$ is
holomorphically isometric to $\CP^n$ with the Fubini-Study metric.
\end{proof}
As an application, we have the following partial result for the equality case of diameter comparison result by   \cite{LW}.
\begin{cor}\label{diamter-t}
Let $M^n$ be a compact K\"ahler manifold with holomorphic
bisectional curvature $\geq 1$.
 Suppose there exist compact connected complex submanifolds   $P$ of dimension  $s$ and $Q$ of dimension $n-1-s$ in $M$ such that $d(P,Q)=\frac{\pi}{\sqrt 2}$ with some $0\leq s\leq n-1$. Then $M$ is holomorphically isometric to $\CP^n$ with the Fubini-Study metric.
\end{cor}
\begin{proof}
Let $[\xi_0:\xi_1:\cdots:\xi_n]$ be the homogeneous coordinate of
$\CP^n$ (equipped with the Fubini-Study metric with constant
holomorphic  sectional curvature 2). Let
$P_0=\{[\xi_0:\xi_1:\cdots:\xi_s:0:\cdots:0]|\xi_0,\cdots\xi_s\in
\C\}\subset \CP^n$ and
$Q_0=\{[0:\cdots:0:\xi_{s+1}:\cdots:\xi_n]|\xi_{s+1},\cdots,\xi_n\in
\C\}\subset \CP^n$. Straight forward computations show:  $d(P_0,Q_0)=d:=\frac{\pi}{\sqrt 2}$ and
\begin{equation}\label{eqn-r_P0-r_Q0}
r_{P_0}(x)+r_{Q_0}(x)=d
\end{equation}
for any $x$, where $r_{P_0}$ and $r_{Q_0}$ are distance functions from $P_0$ and $Q_0$ respectively. Moreover, the first Dirichlet
eigenvalues  of $B_{P_0}(r_0)$ and $B_{Q_0}(d-r_0)$ are both $n+1$. Here for a submanifold $S$ in a K\"aher manifold, $B_S(R)$ consists of points $x$ with  $d(S,x)<R$.

Denote $\lambda_1(N)$  the first Dirichlet eigenvalue of the complex Laplacian of $N$.  By
\cite{MiquelPalmer}, we have
$$
\lambda_1(B_{P}(r_0))\le \lambda_1(B_{P_0}(r_0))= n+1;
$$
$$
\lambda_1(B_{Q}(d-r_0))\le \lambda_1(B_{Q_0}(d-r_0))= n+1.
$$
Since $d(P,Q)=d$, $B_{P}(r_0)\cap B_{Q}(d-r_0)=\emptyset$. By the
same argument as in \cite[Theorem 2.1]{Ch}, we know that the first
nonzero eigenvalue $\lambda$ of $M$ is at most $n+1$. Since the
Ricci curvature of $M$ is bounded below by $n+1$, we have
$\lambda=n+1$. By Corollary \ref{cor-first-eigen},   $M$ is
holomorphically isometric to $\CP^n$ with the Fubini-Study metric.

\end{proof}

\section{Volume comparison}

Let $(M^n,\omega)$ be a compact K\"ahler manifold such that the
first Chern class $c_1(M)$ is either positive or negative, where
$\omega$ is the K\"ahler form that is a multiple of the first Chern
class as a cohomology class. If $c_1(M)<0$ then there is a unique
K\"ahler-Einstein metric $\omega_0$ such that
$\Ric(\omega_0)=-\omega_0$ by well-known results (see
\cite{Aubin,Yau}).

\begin{prop}\label{volume} With the above notations and assumptions, the following are true:
\begin{itemize}
  \item [(i)] If $c_1(M)<0$ and if the scalar curvature $R$ of $\omega$ satisfies
 $-nb\leq R\leq -na<0$, then the volume of $V(M,\omega)$ of $(M,\omega)$ is
 bounded by
 \begin{equation}
 V(M,\frac1b \omega_0)\leq V(M,\omega)\leq  V(M,\frac1a\omega_0).
\end{equation}
If the first (respectively the second) equality holds, then $\omega=\frac1b\omega_0$ (respectively $\omega=\frac1a\omega_0$).
  \item [(ii)] If $c_1(M)>0$ and there is a K\"ahler-Einstein metric $\omega_0$ such that $Ric(\omega_0)=\omega_0$,
      and if the scalar curvature $R$ of $\omega$ satisfies
 $ 0<nb\leq R\leq  na $, then the volume of $V(M,\omega)$ of $(M,\omega)$ is
 bounded by
 \begin{equation}
 V(M,\frac1a \omega_0)\leq V(M,\omega)\leq  V(M,\frac1b\omega_0).
\end{equation}
If the first (respectively the second) equality holds, then
$(M,\omega)$ is holomorphically isometric to  $(M,\frac1a\omega_0)$
(respectively $(M,\omega)$ is holomorphically isometric to
$(M,\frac1a\omega_0)$).
\end{itemize}

\end{prop}
 \begin{proof} We only prove (i), the proof of (ii) is similar (For the case of equality, we have
 used  results in \cite{Bando-Mabuchi}). To prove (i), suppose
Suppose $2\pi c_1(M)=- \lambda[\omega]$ with $\lambda>0$. Then it is
well-known that

\begin{equation}
\lambda=-\frac1{nV(M,\omega)}\int_MR\omega^n.
\end{equation}
 Hence $a\leq\lambda\leq b$.

Moreover,
\begin{equation}
V(M,\omega)= \frac{\pi^n}{n!\lambda^n}\int_Mc_1(M)^n=\frac{V(M,\omega_0)}{\lambda^n}
\end{equation}
Therefore
\begin{equation}
\frac{1}{b^n}V(M,\omega_0)\leq V(M,\omega)\leq \frac{1}{a^n}V(M,\omega_0).
\end{equation}
If the first equality holds, we know that $\lambda=b$ and $\omega$
has constant scalar curvature $-b$ and is K\"ahler-Einstein (see
\cite{Tian}). It is easy to see that $\omega=\frac1b\omega_0$. If
the second equality holds, we can prove similarly that
$\omega=\frac1a\omega_0$.
\end{proof}

\begin{cor}\label{vol-p1} Let $(M^n,\omega)$ be a compact manifold with positive holomorphic bisectional curvature. Suppose the scalar curvature $k_2\ge R\ge k_1>0$. Let $V_{k_1}(\CP^n)$  and $V_{k_2}(\CP^n)$ be the volumes of $\CP^n$ with Fubini-Study metrics with constant scalar curvatures $k_1$ and $k_2$ respectively. Then the volume $V(M,g))$ of $M$ satisfies:
 \begin{equation}\label{vol-e1}
 V_{k_2}(\CP^n)\le V(M,\omega)\le   V_{k_1}(\CP^n).
 \end{equation}
   Moreover,  if one of the inequalities is an equality then $(M,g)$ is holomorphically isometric the $ \CP^n$ with Fubini-Study metric.
\end{cor}
\begin{proof} By \cite{Mori,SiuYau}, $M$ is biholomorphic to $\CP^n$ and we can apply Proposition \ref{volume}(ii) and the results follow.

\end{proof}

\end{document}